\documentclass[12pt]{amsart}
\usepackage{diagbox}
\usepackage{mathtools}
\mathtoolsset{showonlyrefs}
\usepackage{amsmath}
\allowdisplaybreaks
\usepackage{mathdots}

\renewcommand{\emph}[1]{\textit{#1}}
\usepackage{enumerate,amsmath,amsthm,latexsym,amssymb}
\usepackage{color}\usepackage{graphicx}

\definecolor{brown}{cmyk}{0, 0.72, 1, 0.45}
\definecolor{grey}{gray}{0.5}

\newcommand{\old}[1]{}

\newcounter{rot}

\newcommand{\ignore}[1]{}

\def\ii_(#1,#2){i_{#1}^{#2}}
\newcommand{\proj}{\text{proj}}

\def\bx{{\bf x}}
\def\bu{{\bf u}}
\def\by{{\bf y}}

\def\cC{\mathcal{C}}

\def\e{\varepsilon}

\def\k{\kappa}

\def\l{\lambda}
\def\m{\mu}

\def\R{\Rho}

\newcommand{\F}{F}

\def\Z{\mathbb{Z}}

\newtheorem{theorem}{Theorem}[section]
\newtheorem{question}[theorem]{Question}

\newtheorem{lemma}[theorem]{Lemma}
\newtheorem{corollary}[theorem]{Corollary}

\newtheorem{remthm}[theorem]{Remark}

\newtheorem{notation}[theorem]{Notation}
\newtheorem{claim}[theorem]{Claim}

\newenvironment{remark}{\begin{remthm}\rm }{\end{remthm}}%
\newcounter{thmtemp}

\newcommand{\nospace}[1]{}

\def\path{\operatorname{PATH}}

\newcommand{\N}{\mathbb{N}}

\def\rank{\text{rank}}

\newcommand{\ex}{\mathrm{ex}}
\newcommand{\exq}{\mathrm{ex}_q}
\newcommand{\exm}{\mathrm{ex}_\L}
\newcommand{\exmf}{\mathrm{ex}_{\F,\L}}
\newcommand{\exem}{\mathrm{aex}_{\L}}

\newcommand{\bex}{\overline{\mathrm{ex}}}
\newcommand{\bexq}{\overline{\mathrm{ex}}_q}
\newcommand{\Lk}{(\L,\k)}
\newcommand{\lk}{(\L,\k)}
\newcommand{\ls}{(\L,S)}

\newcommand{\supp}{\text{supp}}

\renewcommand{\L}{\mathbf{L}}
\renewcommand{\k}{\mathbf{k}}
\renewcommand{\e}{\mathbf{e}}
\renewcommand{\v}{\mathbf{v}}
\newcommand{\y}{\mathbf{y}}

\renewcommand{\F}{\mathbb{F}}
\renewcommand{\R}{\mathbb{R}}
\newcommand{\cf}{\mathcal{F}}

\newcommand{\prank}{a\textrm{-rank}}
\newcommand{\arank}{a\textrm{-rank}}

\begin{document}
\title{Extremal Collections of $k$-Uniform Vectors}
\author{Joseph Briggs}
\author{Wesley Pegden}\thanks{Research supported in part by NSF grant
    DMS 1700365}

\maketitle

\begin{abstract}
We show any matrix of rank $r$ over $\F_q$ can have $\leq \binom{r}{k}(q-1)^k$ distinct columns of weight $k$ if $ k \leq O_q(\sqrt{\log r})$ (up to divisibility issues), and $\leq \binom{r}{k}(q-1)^{r-k}$ distinct columns of co-weight $k$ if $k \leq O_q(r^{2/3})$. This shows the natural examples consisting of only $r$ rows are optimal for both, and the proofs will recover some form of uniqueness of these examples in all cases.
\end{abstract}

\section{Introduction}

The field of extremal combinatorics deals with the asymptotic study of how parameters grow over increasing classes of discrete structures. Recently (see for example \cite{bonin2003introduction}), there has been growing interest in the study of an extremal theory for matroids. This includes an extremal theory for \emph{representable} matroids, whose ground set is the set of columns of some matrix (and independence is given by linear independence).

One standard method for generating \emph{random} representable matroids, see e.g.\ \cite{cooper2016minors}, is as follows. Construct a matrix representation $M$ by generating $m$ randomly chosen columns of some fixed weight $k$ and length $n$. Indeed, when $k=2$ and the base field is $\F_2$, this gives the graphic matroid of the Erd\H{o}s-R\'enyi random graph $G_{n,m}$ (of which $M$ acts as the vertex-edge incidence matrix). 

Our desire is to settle perhaps the most natural extremal question in this setting: how large can $m$ be, upon fixing the ``size'' of such a representable matroid?
It makes little sense to fix the number $n$ of rows, as then one can take all $m=\binom{n}{k}(q-1)^k$ weight-$k$ column vectors, and every possible matrix will just consist of a subset of these columns. So instead, we fix the rank.

Let us take a step back. For a matrix $M$ over the finite field $\F_q$, we are considering the following question:

\begin{question}
What is the maximum number of distinct columns $M$ can have, if each column has weight $k$, that is $k$ nonzero entries, and $M$ has rank $\leq r$?
\end{question}

We denote this value by $\exq (r,k)$.
We can answer this question if $r$ is large enough:

\begin{theorem}\label{t.mainspecial}
For all $k$ and $q$, there is an $R_{k,q}$ such that for all $r \geq R_{k,q}$, 
    \[
  \exq(r,k)=
  \begin{cases}
    \binom{r+1}{k}
     & q = 2 \text{ and } k \text{ even,}\\
    \binom{r}{k}(q-1)^k & \text{ otherwise.}
  \end{cases}
  \]
\end{theorem}

When $k=q=2$, this tells us that graphs of graphic matroid rank $\leq r$ have $\leq \binom{r+1}{2}$ edges: this was previously noted in Theorem 2.8 of \cite{briggs2017inverting}, where it was shown for \emph{every} $r$ (not just those sufficiently large). Furthermore, the case $q=2$ was a question asked by Ahlswede, Aydinian and Khachatrian \cite{ahlswede2003maximum}. Khachatrian (according to \cite{ashikhmin2005bounds}) and Kramer \cite{kramer2010most} conjectured the above structure, and the latter
proved it when the number of \emph{rows} of the matrix is $r+1$. Our result confirms their conjecture, but only once $r$ is large enough.

The nature of this question does not change much after replacing ``nonzero'' with ``non-$\beta$'' for an arbitrary $\beta \in \F_q$
(see Section \ref{s.weight}, and more specifically Theorem \ref{t.even}, an affine variant we will use to prove this result).
This effectively answers both questions of this type over $\F_2$. However, for other fields $\F_q$, ``weight $k$'' and ``$k\; 1$'s'' have different meanings, suggesting a complementary version of the original question:

\begin{question}
What is the maximum number of distinct columns $M$ can have, if each column has $k$ zeros, and $M$ has rank $\leq r$?
\end{question}

Denoting this by $\bexq(r,k)$,
we have a corresponding result: 

\begin{theorem}\label{t.co}
Suppose $\F_q \neq \F_2$. For all $k$, there is an $\bar{R}_{k,q}$ such that for all $r \geq \bar{R}_{k,q}$,
 \[
  \bexq(r,k)=
    \binom{r}{k} (q-1)^{r-k}.
  \]

\end{theorem}

Furthermore, in the context of both Theorem \ref{t.mainspecial} and Theorem \ref{t.co}, we will show that the only examples attaining the equality have exactly $r$ nonzero rows (unless $k=0$, see Corollary \ref{c.kzero}). This corresponds to the ``uniqueness of the cliques'' in Theorem 2.8 of \cite{briggs2017inverting}, where additional isolated vertices correspond to additional rows of all 0's here.

In fact, a result akin to Theorem \ref{t.mainspecial} holds in a far more general setting. Suppose $\F$ is an arbitrary field (not necessarily finite).
Let $\L = (L_1, \dots, L_s)$ be a collection of disjoint finite sets $L_i \subset \F^*$ of nonzero labels.
 Then, for each $s$-tuple $\k=(k_1, \dots, k_s)$ of positive integers, an ``\emph{$\Lk$-vector}'' is defined to be one with exactly $k_i$ entries in $L_i$ for each $i$, and the rest equal to 0. Thus, a binary vector of weight $w$ is an $\lk$-vector for $\L=(\{1\})$ and $\k=(w)$.
 
The corresponding question in this setting is thus:
\begin{question}
What is the maximum number of distinct columns $M$ can have, if each column is an
 $\Lk$-vector, 
and $M$ has rank $\leq r$?
\end{question}
We denote this value by $\exmf(r,\k)$.

We will prove the following theorem in Section \ref{s.weight}:
 
\begin{theorem}
  \label{t.main}
%
%


Suppose $s=1$, so that $\k=(k)$ and $\L=(L)$. Then there is an $R_k$ such that for all $r \geq R_k$:
\begin{equation}\label{eq.onelist}
\ex_{\F, L} (r, k) =   \begin{cases}
    \binom{r+1}{k} & L = \{\ell\} \text{ and } \ell \cdot 1 =0 \text{ in } \F, \\
    |L|^k \binom{r}{k}& \text{ otherwise.}
  \end{cases}
\end{equation}

Alternatively, suppose each $L_i$ is a single element $\{\ell_i\}$ for every $i$. Then, once $r \geq R_\k$,
\begin{equation}\label{eq.singletonlists}
\ex_{\F,\L} (r,\k) =   \begin{cases}
    \binom{r+1}{\k} & \sum \ell_i k_i =0, \\
    \L^\k \binom{r}{\k}& \text{ otherwise,}
  \end{cases}
\end{equation}
where by $\binom{r}{\k}$ we mean the multinomial coefficient $\binom{r}{k_1,\dots, k_s, r-\sum_i k_i}$, and by $\L^\k$ we mean the product $ \prod\limits_{i \in [s]} |L_i|^{k_i}$.

Moreover, in both cases, any extremal matrix $M$ has only $r+1$ or $r$ nonzero rows respectively.
\end{theorem}

Here, $R_k, R_{\k}$ respectively depend on $|L|$ and $\L^\k$, \emph{but not on the field $\F$}.

We remark there is nonempty (albeit rather small) overlap between Theorem \ref{t.main} and the main theorem of Ahlswede, Aydinian and Khachatrian \cite{ahlswede2003maximum}. They considered this question in the case $\F=\R, s=1$,
and $\L=(\{1\})$, i.e. binary vectors \emph{over the reals} of weight $k$, but managed to solve this for \emph{every} $r$. In particular, the equalities given in \eqref{t.main} were shown to break down precisely once $r < 2k$. As with their question for $q=2$, this leads us to ask how small $R_{k,q}$ can be made in general---we will discuss this a little more in Section \ref{s.conc}. But, for binary vectors over any field, the bound on $R_{k,q}$ we obtain remains the same.

We close the introduction by noting Theorem \ref{t.mainspecial} follows from the first part of Theorem \ref{t.main}. Indeed, let $L:=\F_q^\times = \F_q\backslash \{0\}$, a single list consisting of all nonzero elements of $\F_q$.
Here, $L=\{\ell\} $ if and only if $q=2$ and $\ell=1$, so the clause ``$\ell \cdot 1 =0$'' says precisely that $k$ is even.



\section{Preliminaries, Notation}

%
%
%
%

We first obtain nontrivial bounds for all 3 questions, by generalizing the setup further still.

For an arbitrary set $S\subset \mathbb{Z}_{\geq 0}^{s}$ of possible weight vectors, we
say a column vector is an ``\emph{$\ls$-vector}'' whenever it is an $\lk$-vector for some $\k \in S$, and denote by $\exmf(r,S)$ the maximum size of a collection of $\ls$-vectors whose rank is $\leq r$. We can define $\exq(r,T), \bexq(r,T)$ correspondingly when $T$ is just a subset of nonnegative integers, and specifically write $\exq(r, \leq k), \bexq(r, \leq k)$ as shorthand for $\exq(r, \{0,1,\dots, k\}), \bexq(r, \{0,1,\dots, k\})$ respectively.

We can form a poset structure $\preceq$ on the set of weight vectors $\Z_{\geq 0}^s$ of length $s$ by saying $\k' \preceq \k$ if and only if $k_i' \leq k_i$ in every coordinate $i$. Say that $S \subset \Z_{\geq 0}^s$ is a \emph{down-set} if $\k' \in S$ whenever $\k \in S$ and $\k' \preceq \k$.

\begin{lemma}\label{l.idealbd} For \emph{any} rank $r$, field $\F_q$ and weight $k$:
\[
\bexq(r,\leq k)
= \sum_{i \leq k} \binom{r}{i} (q-1)^{r-i}.
\]
Also, for any field $\F$, down-set $S \subset \Z_{\geq 0}^s$, weight vector $\k$ and list vector $\L$,
\[
\exmf(r, S ) 
=  \sum_{\k' \in S} \binom{r}{\k'} \L^{\k'}, 
\textrm{ and  hence for any }\F_q,
\exq(r, \leq k)
= \sum_{i \leq k} \binom{r}{i} (q-1)^i 
.\]
\end{lemma}

\begin{proof}

It suffices to show ``$\leq$'', since the corresponding lower bounds are all immediate from considering matrices with precisely $r$ rows (with all columns of weight $\leq k$ in the first case, or all $(\L,S)$-vectors of length $r$ in the second).

For the second bound, given a matrix $M$ of rank $r$, let $\cC$ be its columns and $W=\langle \cC \rangle$ be its column space. Since the row rank of $M$ is also $r$, there exists a subset $I$ of $r$ of its rows
such that the projection $W \rightarrow W|_I$ is an isomorphism. In particular, it is injective, and restricts to an injection on the original vectors $\pi: \cC \hookrightarrow \cC|_I$.
For any $ x \in \cC$ and $i \in [s]$, the number of $L_i$-entries in $\pi(x)$ is $\leq$ that of $x$. Hence, if $x$ is an $(\L,\k)$-vector, then $\pi(x)$ is an $(\L,\k')$-vector for some $\k' \preceq \k$, and hence an $(\L,S)$-vector as $S$ is a down-set. The desired bound is then obtained by counting all $(\L,S)$-vectors in $\cC|_I \simeq \F^r$.

The proof of the bound on $\bex$ is identical, with ``zero-entries'' and ``vectors with $k'$ zeros'' in place of ``$L_i$-entries'' and ``$(\L,\k')$-vectors'' respectively.
\end{proof}

\begin{corollary}\label{c.idealbd}
For any $q,\F,r,k,\k$ and $\L$,
\[
\bexq(r,k)  \leq \sum_{i \leq k} \binom{r}{i} (q-1)^{r-i},
\]
\[
\exmf(r,\k) \leq \sum_{\k' \preceq \k} \binom{r}{\k'} \L^{\k'},
\textrm{ and }
\exq(r,k) \leq \sum_{i \leq k} \binom{r}{i} (q-1)^i. 
\]

\end{corollary}

In particular, we obtain the $k=0$ case of Theorem \ref{t.co}:

\begin{corollary}\label{c.kzero}
$\bexq(r,0) = (q-1)^r$. Furthermore, any matrix $M$ with no zeros, of rank $r$, with $(q-1)^r$ distinct columns, has $r$ rows $\bu_1, \dots, \bu_r$ such that every row is a scalar multiple of some $\bu_i$.
\end{corollary}
\begin{proof}
Taking $k=0$ in the 3rd equality of Corollary \ref{c.idealbd} establishes $\bexq(r,0) \leq (q-1)^r$. 

For any matrix $M$ attaining this equality, in the above proof, we see that $\cC|_I$ consists of \emph{all} column vectors with no zeros, that is, $\cC|_I \simeq (\F_q^\times)^r$.
Letting $\bu_1, \dots, \bu_r$ denote the rows of $M$ given by $I$, this says that for every $\v \in (\F_q^\times)^r$ there is some $j$ such that $\v=\left(\begin{array}{c} u_{1,j} \\ \vdots \\ u_{r,j}\end{array} \right)$.
Now, suppose there is another row $u$ of $M$. Since $\rank(M)=r = \dim(\langle \cC|_I \rangle)$ ,
the $\{\bu_i\}$ form a basis for the row space of $M$, so $\bu = \sum \l_i \bu_i$ for some scalars $\l_i \in \F_q$.
 As $M$ has no zeros, $0 \neq \sum \l_i u_{i,j} = \langle \bx, \v \rangle$
 for \emph{every} $j$, writing $\bx:= (\l_1, \dots, \l_r)$.



Now, since $\bu \neq 0$, $\bx \neq 0$ so some $\l_j \neq 0$.
Now take any $j' \in [r] \backslash \{j\}$.
Consider the $q^2$ vectors of the form $\v(\alpha, \beta):=(1,\dots, 1, \alpha, 1,\dots, 1, \beta, 1, \dots, 1)\in \F_q^r$ with $\alpha, \beta$ in positions $j,j'$ respectively.
For each $\alpha \neq 0$, we know $\langle \bx,\v(1,\alpha) \rangle \in \F_q^\times$, and they are distinct since $\l_{j} \neq 0$. It follows $\langle \bx, \v(1,0) \rangle = 0$. As $q \geq 3$, we similarly find another $\beta \in \F_q^\times \backslash \{1\}$, also with $\langle \bx, \v(\beta,0) \rangle = 0$  by the same logic. Subtracting these gives $0=\langle \bx, (1-\beta)\e_{j'} \rangle = (1-\beta)\l_{j'}$, hence $\l_{j'}=0$.

Since $j' \neq j$ was arbitrary, $\bu = \l_j \bu_j$, which is what we were trying to prove.
\end{proof}

\begin{remark}
The final part of the argument showed that any vector  over $\F_q$ ($q \geq 3$) of weight $\geq 2$ is orthogonal to a nonzero number of vectors with no zeros. Later, 
Lemma \ref{l.spacecount} will count this number explicitly.
\end{remark}

\section{Weight-$k$ Proofs}\label{s.weight}

Our proofs will first establish an affine variant of Theorem \ref{t.main} for technical reasons.
To state it, define the \emph{$a$-rank}, or \emph{affine rank}, of a set of vectors to be the smallest $r$ so that any subset of $r+1$ vectors yield an \emph{$a$-dependence}, where by an $a$-dependence we mean a nontrivial linear dependence whose coefficients sum to 0 in $\F$. 

\begin{notation}
We denote by $\exem(r,\k)$ the maximum size of a collection of $\lk$  vectors of $a$-rank $\leq r$. (As $\F$ will always be fixed, we drop the dependence in this notation.)  
\end{notation}

 Notice that, in general, the $a$-rank of a collection is at least the rank of the collection.  On the other hand, the $a$-rank of the columns of a matrix $M$ is the same as the rank of the matrix $M$ with an additional row of all 1's added. 
Thus, we have 
\begin{equation}
\label{sandwhich}
\rank(M)\leq \prank(M)
=\rank \left( \begin{array}{c}
 M  \\
1  \cdots  1
\end{array} \right) \leq \rank(M)+1.
\end{equation}
Moreover,
\begin{remark}\label{rem.arankrank}
Suppose every $L_i = \{\ell_i \}$ has only one element, and that $\sum \ell_i k_i \neq 0$. Then $\exem(r,\k) = \exm(r,\k)$.
Indeed, this time $(1,\dots, 1) \in \text{rowspan}(M)$ for any matrix $M$ whose columns are $(\L,\k)$-vectors, and so its $a$-rank and rank coincide.
\end{remark}


In a similar spirit, we will also make frequent use of the following standard lemma:
\begin{lemma}\label{l.arank}
Let $\l, \m \in \F$ be distinct. 
Then
\[
\prank\left(
\begin{array}{cc}
	\l \cdots \l & \m \\
	Y & \v  
\end{array}
\right) = \prank \binom{\l \cdots \l}{Y} +1
,\]
and hence equals $\prank(Y)+1$ (provided $\l \neq 0$), for any vector $\v$ and matrix $Y$ over $\F$ with the same number of rows.
\end{lemma}

\begin{proof}
It suffices to show ``$\geq$'', the other direction being trivial.\\
Let $r=\arank\left(
\begin{array}{cc}
	\l \cdots \l & \m \\
	Y & \v  
\end{array}
\right)$ and take any $r$ column vectors $\v_1, \dots, \v_r$ of $Y$. By definition of $r$, there is an $a$-dependence among
 $\binom{\l}{\v_1}, \dots, \binom{\l}{\v_r},\binom{\m}{\v}$.
  Since the coefficients sum to 0, and $\l \neq \mu,$ it follows the coefficient of $\binom{\mu}{\v}$ is 0, so in fact we have an $a$-dependence among $\binom{\l}{\v_1}, \dots, \binom{\l}{\v_r}$.
   Thus, $\arank \binom{\l \cdots \l}{Y} \leq r-1$, as desired.
\end{proof}

\begin{theorem}\label{t.even}
Suppose that either every $L_i$ in the list $\L$ has size 1, or that $s=1$ (so $\L=(L_1)$ and $\k =(k_1)$).
Then there is a $Q_\k$ such that for all $r\geq Q_\k$,
  \[
\exem(r,\k)= 
 \begin{cases}
   	\binom{r-1}{k_1} |L_1|^{k_1}  & s=1, |L_1|>1; \\
  	\binom{r}{\k} &
  	\text{ every } |L_i|=1.
 \end{cases}
\]
Moreover, any extremal collection must consist of vectors which are zero except in $r-1$ common positions (respectively, $r$ common positions).
\end{theorem}

\begin{proof}

For the lower bound, we simply take ``all vectors of the maximum possible length''-but we must be cautious whether the maximum possible length is $r$ or $r-1$.
First suppose $L_i = \{ \ell_i \}$ for each $i$.
Let $M$ be the matrix whose columns are all $\binom{r}{\k}$ $\lk$-vectors of length $r$.
Then $\rank(M)=r-1$ if $\sum \ell_i k_i=0$, and $r$ otherwise,
but in both instances $\arank(M)=r$ (see \eqref{sandwhich} and  Remark \ref{rem.arankrank}, respectively).

Meanwhile, if
$\L=(L_1)$ and $|L_1|\geq 2$,
then the matrix of all $|L_1|^{k_1} \binom{r-1}{k_1}$ $\lk$-vectors of length $r-1$ has rank $r-1$, and again by \eqref{sandwhich} has $a$-rank $\leq r$.

Write $\exem^*(r,\k)$ for $\exem(r+ \mathbf{1}_{\{ \exists i: |L_i| > 1 \}},k)$.
For the upper bound, we will prove $\exem^*(r,\k) \leq \L^\k \binom{r}{\k}$ for $r \geq Q_\k$.

To begin, we will show that for any nonzero $\k$
and for all $r \geq \Vert\k\Vert+2$,
  \begin{equation}\label{evenin}
    \exem^*(r,\k)\leq  \exem^*(r-1,\k) + \sum_{i \in [s]} |L_i|\cdot \exem^*(r-1,\k-\e_i).
  \end{equation}
  where $\e_i$ denotes the $i$th unit vector, so that $\k-\e_i=(k_1, \dots, k_i-1, \dots, k_s).$
  To see this, we first show \eqref{evenin} without the stars. Consider any matrix $M$  of $a$-rank $\leq r$ whose columns are all $\lk$-vectors. 
If all nonzero rows of $M$ had all entries in $\bigcup L_i$, then $M$ has only $\Vert\k\Vert$ nonzero rows and in particular $\leq \L^\k$ columns, independently of $r$. Plus, having already established the lower bound in the theorem, we know $\L^\k \leq \binom{r-2}{\k} \L^\k \leq \exem(r-1,\k)$, only needing $r \geq 2 + \Vert\k\Vert$. So WLOG, assume that the first row of $M$ contains both a 0 and an $\ell \in \bigcup L_i$.

Now let $A_\ell$ be the set of vectors with $\ell$ in row 1, for each $\ell \in \{0\} \cup \bigcup L_i$. Both $\bigcup_{\ell \neq 0} A_\ell$ and $A_0$ are nonempty by assumption.  Define $A'_\ell$ to be the collection of vectors produced by removing the first coordinate from each vector in $A_\ell$. 
%
By Lemma \ref{l.arank}, $A_\ell'$ has $a$-rank $\leq r-1$ for every $\ell \in \{ 0 \} \cup \bigcup L_i$. Hence $|A_\ell'| \leq \exem(r-1,\k-\e_i)$ whenever $\ell \in L_i$ while $|A_0'| \leq \exem(r-1,\k).$ This establishes \eqref{evenin} for $\exem(r,\k)$. To add the stars, note that the shift in $r$ of 1 in $\exem^*$ occurs for \emph{all} terms in the case $s=1, |L_1|>1$ and for \emph{no} terms if $|L_i|=1$ for every $i$.


The inequality \eqref{evenin}  would suffice to prove Theorem \ref{t.even} if we could establish a family of base cases for each nonzero $k$ for the induction, since clearly $\exem^*(r,\mathbf{0})=1$ for every $r$ and $\L$. We do not know how to do this directly, however.  Instead we define
\[
\alpha^\k_r=\exem^*(r,\k)-\L^\k\binom{r}{\k}
\]
and consider the sequence $\{\alpha_r^\k\}_{r\in \N}$. Now, \eqref{evenin} gives that for $r \geq \Vert\k\Vert+2$, $\alpha_r^\k\leq \alpha_{r-1}^\k + \sum\limits_{i \in [s]} |L_i| \cdot \alpha_{r-1}^{\k-\e_i}$.  By induction on $\Vert\k\Vert$, we then have for $Q_\k':= \max \{ \Vert\k\Vert+2\} \cup \{Q_{\k-\e_i}:i \in [s]\}$ that 
\[\label{monotone}
r-1\geq Q_\k'\implies\alpha_{r}^\k \leq \alpha_{r-1}^\k.
\]
Observe that to prove Theorem \ref{t.even}, it suffices to show that for $r-1\geq Q_\k'$,

\begin{claim}\label{c.strict}
\begin{equation}
\alpha_{r}^\k=\alpha_{r-1}^\k\implies (\mbox{$\alpha_r^\k=0$ and any collection realizing $\exem^*(r,\k)$ must have support $r$}).
\end{equation}
\end{claim}

To prove Claim \eqref{c.strict}, let us suppose that $r,\k$ are such that $\alpha_r^\k=\alpha_{r-1}^\k$, and $r-1\geq Q_\k'$, so that 
\begin{equation}
\label{equality}
\exem^*(r,\k)
=\exem^*(r-1,\k) + \sum_{i \in [s]}|L_i|^{k_i}\binom {r-1} {\k-\e_i}.
\end{equation}
Recall that in the decomposition above, $A_0$ has size at most $\exem^*(r-1,\k)$.  Thus for an extremal collection for $r,\k$ where \eqref{equality} holds, we have that 
\begin{equation}
\label{eq.nonzerosubmxs}
\sum_{\ell \neq 0} |A_\ell|
\geq
\sum |L_i|^{k_i} \binom{r-1}{\k-\e_i}.
\end{equation}
 Moreover, by induction on $\Vert\k\Vert$, we have that the unique candidate for $A_\ell'$ of size $\L^{\k-\e_i}\binom{r-1}{\k-\e_i}$ is a collection of vectors whose support has size $r-1$, for every $\ell \in L_i$.
In fact, they are all $\L^{\k-\e_i}\binom{r-1}{\k-\e_i}$ such possible vectors.  The above inequality \eqref{eq.nonzerosubmxs} is therefore an equality, already establishing $\alpha_r^{\k} = 0$.

We next see these supports must be identical for every $\ell \in \cup L_i$. For otherwise, for some distinct $\ell,j \in \cup L_i$ we have $\supp(A_\ell') \backslash \supp(A_j') \neq \emptyset$, WLOG containing $2$. That is, some $u' \in A_j'$ has $u'_2 \in \cup L_i$ while 2 is outside the support of $A_j'$. Write $A_j''$ for $A_j'$ upon deletion of this topmost row of all zeros.
Then take
some other $v' \in A_\ell'$ with $v'_2 =0$ (using the structure of $A_\ell'$ and that $r > \Vert \k \Vert$).
Now expand along rows 1 and 2 in Lemma \ref{l.arank} in turn, writing $v'',u''$ for the vectors obtained from $v',u'$ respectively by removing the first entry:
\begin{align*}
\arank(A_j')
 =\arank(A_j'') 
 \leq \arank\left( \begin{array}{cc}
	j \cdots j & \ell \\
	A_j'' & v''
	 \end{array} \right)-1
& \leq \arank\left( \begin{array}{ccc}
	j \cdots j & \ell & \ell \\
	0 \cdots 0 & 0 & u_2 \\
	A_j'' & v''
	& u'' \\
	 \end{array} \right)-2 \\
&  \leq r-2,
\end{align*}
 a contradiction.

Finally, we check the support of $A_0$ must be contained in the support of $A_j'$ (now equivalent for any $j \in \cup L_i$), establishing Claim \eqref{c.strict}, and thus also Theorem \ref{t.even}.  Indeed, suppose $A_0$ contains a vector $u$ such that $u(t) \in \cup L_i$, where $t$ is outside the support of $A_j'$. This time we consider three cases:\\
\textbf{Case 1: All vectors $v\in A_0$ satisfy $v(t) \in \cup L_i$, but $A_0$ is nonconstant on row $t$.}  Decomposing $A_0$ according to $t$-th entries, and applying Lemma \ref{l.arank} to each part we see $|A_0|\leq \sum |L_i| \cdot \exem^*(r-1,\k-\e_i)$. So by induction on $k$, $|A_0|\leq \sum |L_i|\cdot \L^{\k-\e_i} \binom{r-1}{\k-\e_i}=O(r^{\Vert\k\Vert -1})<\exem^*(r-1,\k)$ for $r$ large enough, but this contradicts \eqref{equality}.\\
\textbf{Case 2: All vectors $v\in A_0$ satisfy $v(t) =\ell \in L_i$.} Deleting row $t$ from $A_0$ then does not affect the $\arank$, so in fact $|A_0| \leq \exem(r-1,\k-\e_i) \leq O(r^{\Vert\k\Vert-1})$ by induction, again a contradiction.
\\
\textbf{Case 3: There is a vector $v\in A_0$ with $v(t) =0$.}  In this case, 
two applications of Lemma \ref{l.arank}, using rows $t$ and 1 in turn, show $\arank(M) \geq \arank(u \vert v \vert A_j')=\arank(v \vert A_j')+1 = \arank(A_j')+2=r+1,$ a contradiction.
%
\end{proof}

\begin{remark}\label{r.bound}
We can obtain an explicit bound on $Q_\k$ as follows. Claim \eqref{c.strict} was sufficient to prove the theorem since $\{\alpha_r^\k\}$ is bounded below by 0. But in fact, by recalling $\rank(M) \leq \arank(M)$ and applying Corollary \ref{c.idealbd},
\begin{align*}
\alpha_{Q_{\k}'}^\k \leq \exem^*(Q_{\k}',\k) \leq \exm(Q_{\k}'+1,\k)
& \leq
\sum_{\k' \preceq \k} \binom{Q_{\k}'+1}{\k'} \L^{\k'}\\
& \leq 
\binom{Q_{\k}'+1}{\k} \L^\k \sum_{\k' \preceq \k}  \left(\frac{\Vert\k'\Vert}{r}\right)^{\Vert\k-\k'\Vert}\\
& \leq \binom{Q_\k'+1}{\k} \L^\k
\sum_{i=0}^{\Vert\k\Vert} \binom{\Vert\k||}{i} \left( \frac{\Vert\k\Vert}{r}\right)^i \\
& = \binom{Q_\k'+1}{\k} \L^\k \left( 1+\frac{\Vert\k\Vert}{r} \right)^{\Vert\k\Vert}\\
& \leq O\big((Q_{\k}' \vert \L \vert)^{\Vert\k\Vert}\big),
\end{align*} 
e.g.\ taking $\vert \L \vert:=\max_i \{|L_i|\}$.
So the decreasing sequence $\{\alpha_r^\k\}$ stabilizes after $\leq O\big((Q_{\k}'\vert \L \vert)^{\Vert\k\Vert}\big)$ additional steps. Thus we can take $Q_\k:= Q_\k' + O\big((Q_\k'\vert \L \vert)^{\Vert\k\Vert} \big)$ in the theorem.
This way, $Q_\k$ is bounded by $(\vert \L \vert+1)^{O(\Vert\k\Vert^2)}$ as $\Vert\k\Vert \rightarrow \infty$.

\end{remark}

We are now ready to prove Theorem \ref{t.main}, and make the transition from affine rank to usual rank. 

\begin{proof}[Proof of Theorem \ref{t.main}]

Note that if every $L_i=\{ \ell_i\}$ and $\sum \ell_i k_i \neq 0$, the $a$-rank and rank coincide, and we are immediately done by Theorem \ref{t.even}.
%

So next suppose some $|L_i|>1$. If we take the $\L^\k \binom{r}{\k}$ $\lk$-vectors with some fixed support of size $r$, then the rank is exactly $r$. This gives the lower bound.

Now consider any collection of $\lk$-vectors of rank at most $r$.  By \eqref{sandwhich}, the $a$-rank is $\leq r+1,$ and Theorem \ref{t.even} shows the size is at most $\L^\k \binom{r}{\k}$, along with the uniqueness of the equality case.

Lastly we consider the case where $\forall i L_i = \{\ell_i\}$ but $ \sum \ell_i k_i =0$.  

For the lower bound, if we now take the $\L^\k\binom{r+1}{\k}$ $\lk$-vectors with some fixed support of size $r+1$, then the rank is at most $r$, since they all lie in the subspace $x \cdot (1,\dots, 1)=0$.

For the upper bound, any collection of $\lk$-vectors of rank at most $r$ has $a$-rank at most $r+1$, and we finish by Theorem \ref{t.even} again, this time concluding they number $\leq \L^\k \binom{r+1}{\k}$.

\end{proof}

We may generalise Theorem \ref{t.main} to an arbitrary set $S\subset \mathbb{Z}_{\geq 0}^{s}$ of possible weight vectors as follows. Recall that a vector is an $\ls$-vector whenever it is an $\lk$-vector for some $\k \in S$, and that $\exmf(r,S)$ is the maximum size of a collection of $\ls$-vectors whose rank is $\leq r$. We say the pair $(\L,\k)$ is \emph{shifted} if every $|L_i| = 1$, say $L_i=\{\ell_i\}$, and $\sum \ell_i k_i =0$ in $\F$. Thus, every $(\L,\k)$-vector automatically lies in the hyperplane $\sum x_j = 0$.
\begin{corollary}
Suppose that either $|L_i|=1$ for every $i$, or that $s=1$.

Then there is an $R_S$ such that for all $r \geq R_S,$
\[
\exmf(r,S) = 
	\begin{cases}
		\sum\limits_{\k \in S}  \binom{r+1}{\k} &
		\forall \k \in S\; (\L,\k) \text{ is shifted}, \\
		\sum\limits_{\k \in S} \L^\k \binom{r}{\k} & 
		\forall \k \in S,\; (\L,\k) \text{ is not shifted.}
	\end{cases}
\]

Moreover, any extremal matrix $M$ has only $r+1$ or $r$ nonzero rows respectively.
\end{corollary}

\begin{proof}

The equalities follow directly from Theorem \ref{t.main}, simply by decomposing a given collection of $(\L,S)$-vectors into respective $(\L,\k)$-vectors.

Now, let $M$ be any extremal matrix, and for each $\k \in S$, write $A_{\k}$ for its submatrix of $(\L,\k)$-vector columns. By the equality case of Theorem \ref{t.main}, each $A_\k$ has only $r+1$ or $r$ nonzero rows (in the respective cases). These supports are then identical, for otherwise $\rank(M) \geq r+1$ is witnessed by adding any $(\L,\k')$-vector $v$ with a nonzero entry outside of the support of $A_{\k}$ to $A_{\k}$.
\end{proof}

One may hope to combine the shifted and not-shifted vectors together in $S$, but having exactly $r$ rows appears not to always be optimal in this case. Indeed, consider vectors of weight 1 or 4 over $\F_2$ (so $\L=(\{1\})$, $S=\{1,4\}$).
Then one may vainly hope all $\binom{r}{4} + \binom{r}{1}$ possible $(\L,S)$-vectors in a matrix with exactly $r$ rows is optimal. But this can be improved to $\binom{r+1}{4}$ columns in a matrix with $r+1$ rows, by restricting to just the vectors of weight 4: without any weight-1 vectors, all columns will lie in the hyperplane $\sum x_j =0$, reducing their rank from $r+1$ to $r$.

\section{$k$ Zeros Proofs}\label{s.coweight}

We now proceed with the proof of Theorem \ref{t.co}. First, we establish a standard counting function:

\begin{lemma}
For any sequence $u_1, u_2, \dots$ of \emph{nonzero} elements of $\F_q$, and any number $n$,
the number of vectors $x \in (\F_q^\times)^n$ orthogonal to $(u_1, \dots, u_n)$ is
$a^{(0)}_n=\frac{1}{q}\big( (q-1)^n + (-1)^{n} (q-1) \big).$
\end{lemma}

\begin{proof}
More generally, let $a^{(\beta)}_n:=|S_n^{\beta}|$, where $S_n^\beta :=\{x \in (\F_q^\times)^n: x_1u_1+\ \dots + x_nu_n=\beta\}$ . Since $x \mapsto \beta x$ is a bijection $S_n^{1}\rightarrow S_n^\beta$ for every $\beta \in \F_q^\times$, it follows $a^{(\beta)}_n = a^{(1)}_{n}$.
Furthermore, $|S^\alpha_{n+1}| =\sum_{\beta \neq \alpha} |S^\beta_n|$ for any $\alpha \in \F_q$, since any vector in $\bigsqcup_{\beta \neq \alpha} S^\beta_n $ can be uniquely extended to a vector in $S^\alpha_n$, since $v_n \neq 0.$ Thus, we have the recursive relations for each $n \geq 0$:
\begin{align*}
a^{(0)}_{n+1} &= (q-1)a^{(1)}_{n},
\\
a^{(1)}_{n+1} &= (q-2)a^{(1)}_{n} + a^{(0)}_{n}.
\end{align*}
Since $a_0^{(\beta)} = \mathbf{1}_{\beta = 0}$, the results $a_n^{(0)} = \frac{1}{q}\big( (q-1)^n + (-1)^{n} (q-1) \big)$ and
$a_n^{(1)}= \frac{1}{q}\big( (q-1)^n + (-1)^{n+1} \big)$ follow by a trivial induction (or may be derived directly using generating functions).
\end{proof}

For a fixed $r$, we use $X$ as shorthand for $\F_q^r$. Furthermore, for each $n \leq r$, we denote by $X^{\geq n}$ and $X^{=n},$ the sets of vectors of weight $\geq n$ and exactly $n$ respectively. Immediately note that $|X^{=n}| = \binom{r}{n}(q-1)^{n}$ for every $n$.

\begin{lemma}\label{l.spacecount}
Suppose $\v \in X^{\geq 2}$ has $i \geq 2$ non-zero entries, and $W$ is its orthogonal complement $W:=\v^\perp = \{\bx \in X: \bx \cdot \v = 0\}.$
Then 
 \begin{align*}
 |X^{=r-k} \cap W| 
 & = \frac{1}{q}
 \left(
\binom{r}{k} (q-1)^{r-k}+
(-1)^{i}(q-1)^{r-i-k+1}
\sum_{s=0}^k (1-q)^s \binom{i}{s} \binom{r-i}{k-s}
  \right) \\
 & \geq \frac{1}{q} \binom{r}{k}(q-1)^{r-k}\left( 1
 -\frac{1}{(q-1)}
 \right) \text{ for } r \text{ sufficiently large.}
 \end{align*}
\end{lemma}

\begin{proof}
WLOG, $\v=(v_1, \dots, v_i, 0, \dots, 0)$ where $v_1, \dots, v_i$ are all non-zero.

For each $S \in \binom{[r]}{k},$ let $W_S:= \{\bx \in X^{=r-k} \cap W: \{j\in [r] : x_j = 0\} = S \}$, so we may decompose $W$ as $\bigcup_{s=0}^k
\bigcup_{|S \cap [i]| = s} W_S$.

We claim that, if $|S \cap [i]| = s$, then 
$|W_S| = \frac{1}{q} \left(
(q-1)^{r-k} +(-1)^{i+s}(q-1)^{r-i-k+s+1}   
\right)$. Since there are $\binom{i}{s} \binom{r-i}{k-s}$ such $S \in \binom{[r]}{k}$ with $|S \cap [i]|=s$, the above decomposition gives the result.

To see the claim, note
\[
\bx \in W_S \Leftrightarrow 
\left\lbrace \begin{array}{ccc}
x_j = 0 & \forall j \in S\\
x_j \neq 0 & \forall j \in [r] \backslash S\\
\sum\limits_{j \in [r] \backslash S } x_j v_j =0 &
\end{array} \right.
\Leftrightarrow 
\left\lbrace \begin{array}{ccc}
x_j = 0 & \forall j \in S & \\
x_j \neq 0 & \forall j \in [r] \backslash (S \cup [i])&\\
x_j \neq 0 & \forall j \in [i] \backslash S &\wedge \sum\limits_{j \in [i] \backslash S } x_j v_j =0. 
\end{array} \right.
\]
Applying the lemma to $(u_1, \dots, u_n):= \proj_{[i]\backslash S}(\v)$ and noting $n=i-s$, we see there are $\frac{1}{q}\big( (q-1)^{i-s} + (-1)^{i-s} (q-1) \big)$ ways to choose the entries of $\bx$ in $S \cap [i]$.
Furthermore, there are $(q-1)^{|[r] \backslash (S \cup [i])|} = (q-1)^{r-k-i+s}$ ways to choose the entries of $\bx$ in $[r] \backslash (S \cup [i]),$ and so there are
$\frac{1}{q} \left(
(q-1)^{r-k} +(-1)^{i+s}(q-1)^{r-i-k+s+1}   
\right)$ such $\bx$ in total.

We now proceed to prove the claimed inequality.
Let $a_s := \binom{i}{s}\binom{r-i}{k-s}(q-1)^s$,
for each $0 \leq s \leq k$.

Note that, for $s < i,$ $\frac{a_{s}}{a_{s+1}}
= \frac{(s+1)(r-k-i+s+1)}{(i-s)(k-s)(q-1)}$ is an increasing function in $s$ (and for $s>i$ $a_s=0$ anyway). Hence the sequence $\{a_s\}$ is \emph{unimodal}, i.e.\ consists of a (possibly empty) monotonically increasing subsequence followed by a decreasing subsequence.
In particular, the alternating sum
$\sum_{s=0}^k (-1)^s a_s $ is bounded above by $\max_s \{a_s\}.$ Let us fix the $s$ attaining this maximum.

Now, 
\[\frac{\binom{r}{k} (q-1)^{i-1}}{a_s}
= (q-1)^{i-1-s}\frac{\binom{r}{k}}{\binom{i}{s}\binom{r-i}{k-s}} \geq q-1
\] provided $i \geq s+2$, since the denominator is a single term in the identity
$\sum_{s'} \binom{i}{s'} \binom{r-i}{k-s'} = \binom{r}{k}$. Else, $i \in \{s,s+1\}.$ We check the lower bound still holds here:

When $i=s+1$, the above is
$ \frac{\binom{r}{k}}{i\binom{r-i}{k-i+1}}
\geq \frac{\binom{r}{k}}{i\binom{r-i}{k-i}}
= \frac{\binom{r}{i}}{i\binom{k}{i}}
\geq \frac{(r-1)^2}{i(k-1)^2}
$
(using $i \geq 2$ and e.g.\  $r \geq 2k$).

Similarly, if $i=s$, the above is 
$\frac{\binom{r}{k}}{(q-1)\binom{r-i}{k-i}}
= \frac{\binom{r}{i}}{(q-1)\binom{k}{i}}
\geq \frac{(r-1)^2}{(q-1)(k-1)^2}$.

So these are both still $\geq q-1$, assuming $r \geq \max \{q^{1/2}k^{3/2},qk\}.$ In summary,
\begin{align*}
\frac{1}{q} (q-1)^{r-i-k+1}
\left(
\binom{r}{k}(q-1)^{i-1}
+\sum_{s'=0}^k (-1)^{s'+i} a_{s'}
\right)
&\geq 
\frac{1}{q} (q-1)^{r-i-k+1}
\left(
\binom{r}{k}(q-1)^{i-1}
- a_s
\right) \\
& \geq 
\frac{1}{q} \binom{r}{k}(q-1)^{r-k}
\left(
1-\frac{1}{q-1}
\right).
\end{align*}

\end{proof}

Our remaining tool is a standard observation in abstract linear algebra. For a vector space $V$ over $\F_q$, recall that the \emph{dual space} $V^*$ consists of all linear functions $V \rightarrow \F_q$, and has the same dimension as $V$ (when finite).

\begin{lemma}
Suppose $f_1, \dots, f_r \in V^*$ and $\bigcap_{j=1}^r V_j = \{0\}$. Then $f_1, \dots, f_r$ are linearly independent.
\end{lemma}

We are now in a position to prove the nonzero case of Theorem \ref{t.co}. This time, the extremal matrices cannot have \emph{any} duplicate rows, nor scalings thereof.

\begin{theorem}
Let $k \geq 1$.
Then
$\bexq (r,k) = \binom{r}{k} \cdot (q-1)^{r-k}$,
 provided $r\geq \max\{3q^2k,q^{1/2}k^{3/2}\}$. Furthermore, the unique extremal example is a matrix $M$ consisting of only $r$ rows and all possible columns.
\end{theorem}

\begin{proof}
We may assume $\rank(M)=r$, and that all rows are distinct.
Let $Y$ denote the set of columns of $M$, with span $ \langle Y \rangle = V$. If $r'$ denotes the number of rows of $M$, then $V \leq \F_q^{r'}$ is a subspace of dimension $r$.

For each $j \in [r']$, we have $V_j':= \{ \y \in \F_q^{r'}:y_j=0\}$ is codimension-1 in $\F_q^{r'},$ and hence $V_j := V_j' \cap V$ is codimension-$\leq 1$ in $V$. Whenever $\dim(V_j)=r-1$, we say that row $j$ is \emph{nontrivial}, and observe
$V_j =\{\y \in V: f_j(\y)=0\}$ for some $f_j \in V^*$, the dual space of $V$ (namely, linear functions $V \rightarrow \F_q$).
 Otherwise, $V_j = V$, and we say row $j$ is \emph{trivial}. (In fact, every trivial row of $M$ is necessarily all zeros, but we will not need this for the argument.)

In fact, since the row rank of $M$ is $r$, we see there are some $r$ rows $I$ which are linearly independent. Thus, the projection onto just these coordinates is a linear isomorphism $\Pi_I:V \rightarrow \F_q^I (\simeq \F_q^r)$.
In this way, for every nontrivial $j$, $\Pi_I(V_j) \leq \F_q^I$ is also a codimension-1 subspace.

  We have by assumption that every $\y \in Y$ is in \emph{exactly} $k$ of the $\{V_j\}$ counting multiplicities, and hence in $\kappa$ of the $\{V_j : j \text{ nontrivial} \}$, where $\kappa: = k - |\{j \text{ trivial}\}|$.

As such, for every $\y \in Y$, $\Pi_I(\y)$ is in exactly $k$ of the mapped subspaces $\cf:=\{\Pi_I(V_j): j \text { nontrivial} \}$
(viewed as a multiset).
In particular, $\cf$ contains every coordinate subspace $\e_\ell^\perp$ \emph{at least once}
(as $\e_\ell^\perp = \Pi_I(V_\ell)$ for each $\ell \in I$).
 Every $\Pi_I(\y)$ is in $\leq \kappa$ of these coordinate subspaces, and hence has $\leq \kappa$ zeros. Deduce
$\Pi_I(Y) \subset X^{\geq r-\kappa}$, so we immediately obtain $|Y|=|\Pi_I(Y)| \leq 
\binom{r}{\kappa}(q-1)^{r-\kappa} + \binom{r}{\kappa-1}(q-1)^{r-\kappa+1}+ \dots + (q-1)^r$.
Of course, that was something we already established in Corollary \ref{c.idealbd}, but we will need this setup to help remove the trailing terms.

Suppose first that there is some $W \in \cf$ which is not a coordinate hyperplane.
We will show that $|Y|$ is too small in this case. Now, by dimension counting, $W^\perp = \langle \v \rangle$ for some $\v \in X$. Plus, as $W$ is not a coordinate hyperplane, $\v$ has $\geq 2$ non-zero entries. Thus, the previous lemma shows there are many vectors of weight $r-\kappa$ in $W$.

In fact, $\Pi_I(Y) \subset (X^{=r-\kappa} \backslash W) \cup (X^{\geq r-\kappa+1})$, since all vectors in $X^{\leq r-\kappa-1} \cup (X^{=r-\kappa}\cap W)$ are in $\geq \kappa+1$ spaces in $\cf$.
Also note that $|X^{\geq r-\kappa+1}| \leq 2 |X^{={r-\kappa+1}}|
= 2\binom{r}{\kappa-1}(q-1)^{r-\kappa+1}$ provided $r\geq 2q \kappa $. 
Putting these together with Lemma \ref{l.spacecount},
\begin{align*}
|Y|=|\Pi_I(Y)|
& \leq |X^{=r-\kappa}| - |X^{=r-\kappa} \cap W| + | X^{\geq r-\kappa+1}|\\
&\leq \binom{r}{\kappa}(q-1)^{r-\kappa}
-\frac{1}{q} \binom{r}{\kappa}(q-1)^{r-\kappa} \left( 1-\frac{1}{(q-1)}\right)
+2 \binom{r}{\kappa-1}(q-1)^{r-\kappa+1}\\
& < \binom{r}{\kappa}(q-1)^{r-\kappa},
\text{ if } r \geq 3q^2 \kappa.
\end{align*}

As such, we may assume every $\Pi_I(V_j) \in \cf$ is some coordinate hyperplane $\e_\ell^\perp$. However, we still are not yet sure that the original subspaces $\{V_j\}$ were distinct (in the way that the $\{V_j'\}$ are): there may be collisions upon intersection with $V$. 

For each $\bx \in \F_q^r$, we denote by $Z_\bx$ its zero-set $\{\ell: x_\ell = 0\}$. Also, letting $w(\bx):= |\{W \in \cf: \bx \in W\}|$ (counting multiplicities), we see that every $\bx \in \Pi_I(Y)$ has $w(\bx)=\kappa$.
Form a poset structure on $X=\F_q^r$ by $\bx \prec \by \Leftrightarrow Z_\bx \supsetneq Z_\by$: thus, $(X, \preceq)$ looks like a blowup of the Boolean lattice where each vector of weight $n$ has been blown up $(q-1)^n$ times. Also, since 
$\cf$ contains each coordinate subspace at least once, $w$ is a strictly increasing function on $(X,\preceq)$, 
and hence $\Pi_I(Y)$ forms an antichain.

This satisfies a LYM-type inequality (see e.g.\ \cite{bollobas1986combinatorics} for an exposition we will mimic here):
for any arbitrary $A\subset X$, write $A^{=i}:=A \cap X^{=i}$ for each $i \leq r$. Then a random maximal chain $C$ in $X$ satisfies $\mathbb{E}[|C \cap A|] = \sum_{i \leq r} \frac{|A^{=i}|}{|X^{=i}|}$ by symmetry. For the antichain $A:=\Pi_I(Y)$, deduce this is $\leq 1$. Furthermore, with $r \geq qk \geq q \kappa$, we have $|X^{=r-\kappa}|> |X^{=r-\kappa+1}|> \dots > |X^{=r}|$, and hence 
\[
1\geq \sum_{i \leq r} \frac{|A^{=i}|}{|X^{=i}|}
= \sum_{r-\kappa \leq i \leq r} \frac{|A^{=i}|}{|X^{=i}|}
\geq \sum_{r-\kappa \leq i \leq r} \frac{|A^{=i}|}{|X^{=r-\kappa}|}
=\frac{|A|}{|X^{=r-\kappa}|},
\]
so $|Y| =|A| \leq |X^{=r-\kappa}| = \binom{r}{\kappa}(q-1)^{r-\kappa} \leq \binom{r}{k}(q-1)^{r-k}$ is immediate. In the equality case, $k = \kappa$, and all rows were nontrivial. Also, every $\frac{|A^{=i}|}{|X^{=i}|} = \frac{|A^{=i}|}{|X^{=r-k}|}$ for $i > r-k$, hence they are all 0, from which it follows $\Pi_I(Y)=A=X^{=r-k}$.

Deduce $\cf=\{\e_1^{\perp},\dots, \e_r^\perp \}$ with no repeated subspaces, so $r'=r$ and $M$ only had $r$ rows originally.
\end{proof}

\section{Concluding Remarks and Further Questions}\label{s.conc}


In light of Theorems \ref{t.mainspecial} and \ref{t.co}, one may naively hope that $\bexq(r,k) = \exq(r,r-k)$ for some reasonable values of $r,k$ and $q$, but this is very far from being true, so there is a limit to how small we can make $R_{k,q}$ and  $\bar{R}_{k,q}$. Indeed, $\bexq(r,k)$ is an increasing function of $k$ (for fixed $r$ and $q$), since adding a row of zeros does not increase the rank of a matrix. Plus, while adding rows of all $1$'s might increase the rank, it does not increase the $a$-rank, so $\exem(r,\k)$ is also an increasing function of $\k$.

Even more strikingly,
$\exq(r,k) = \bexq(r,k)=0$ for negative $k$, whereas $\exq(r,k), \bexq(r,k)$ can be defined for $k >r$ and are clearly positive: in fact, $\exq(r,(q-1)q^{r-1})= \bexq(r,q^{r-1}) = q^r-1$. This is clearly the most possible for \emph{any} $k$, since an $\F_q$-vector space of dimension $r$ only has $q^r$ distinct elements in total, including $0$. 

The matrix $M$ attaining the above is simply the \emph{dual Hamming code} \cite{hamming1950error} of length $q^r$, as noted in the concluding section of Ahlswede, Aydinian and Khachatrian \cite{ahlswede2003maximum} (and in fact, was shown to be essentially the unique such matrix up to repetition by Bonisoli \cite{bonisoli1984every}).
Explicitly, we list all $q^r$ vectors $\F_q^r=\{\v_1, \dots, \v_{q^r}\}$
as the \emph{rows} of a matrix $A$,
then let the columns of $M$ consist of all nonzero vectors in the column space of $A=(\bu_1 | \dots | \bu_r)$,
so $\rank(M) =r$.
Now, the $i$-th entry of a column $\sum \l_j \bu_j$ of $M$ is zero if and only if $\v_i$ is in the hyperplane $\{\sum_j \l_j x_j = 0\}$. This is true for exactly $q^{r-1}$ such vectors $\v_i \in \F_q^r$, and hence every column of $M$ has weight $(q-1)q^{r-1}$.

So, we know these theorems cannot be extended arbitrarily. But we can still ask about the threshold functions:

\begin{question}
How small can $R_{k,q}$ and $\bar{R}_{k,q}$ be made in Theorems \ref{t.mainspecial} and \ref{t.co}?
\end{question}

Theorem \ref{t.co} was established directly, obtaining the result for $\bar{R}_{k,q} = O_q(k^{3/2})$. In sharp contrast, the proof of Theorem \ref{t.mainspecial} used an induction for which we were unable to directly establish a base case, and is only known for $R_{k,q} = 2^{O_q(k^2)}$.

Now, $\bar{R}_{k,q}$ can't be made independent of $q$. Once $r < qk$, we note that $\binom{r}{k}(q-1)^{r-k} < \binom{r}{k-1}(q-1)^{r-k+1}$, so (for example) our usual example is beaten by a matrix consisting of all co-weight $k-1$ vectors of length $r$, and then appending a row of all $0$'s. Perhaps it is still true that $\bexq(r,k) = \binom{r}{k'}(q-1)^{r-k'}$ for some $k'<k$, using matrices with lots of empty rows. Similar logic shows that $R_{k,q}$ can't be made smaller than $\frac{q}{q-1}k$.
Yet, it is still plausible that e.g.\ $\ex_2(2k,k) = \binom{2k}{k}$ for every odd $k$, and even that $R_{k,q}$ can be made independent of $q$.

Furthermore, we wonder whether Theorem \ref{t.co} can be generalized in a similar fashion to Theorem \ref{t.main}. 
This leads to questions that lie strictly between the original two, the simplest instance of which is the following:
\begin{question}
Does every rank-$r$ matrix over $\F_4$ have $\leq \binom{r}{k} \cdot 2^r$ columns with exactly $k$ entries either 0 or 1 (for all $r$ sufficiently large)?
\end{question}

%
%
%
%
%
%
%
%
%
%
%
%
%
%
%
%
%
%
%
%
%
%
%
%
%
%
%
%
%

\section{Acknowledgements}

We would like to thank Boris Bukh for helpful discussions and for suggestions significantly improving a previous draft of the paper, and Imre Leader for informing us of the relevance of 
\cite{ahlswede2003maximum}.
We are also indebted to an anonymous referee for their very careful reading and finding of a number of mistakes in an earlier version of the paper.

\bibliographystyle{abbrv}
\bibliography{ExtremalMatroidsReferences}

\begin{thebibliography}{1}

\bibitem{ahlswede2003maximum}
R.~Ahlswede, H.~Aydinian, and L.~Khachatrian.
\newblock Maximum number of constant weight vertices of the unit n-cube
  contained in a k-dimensional subspace.
\newblock {\em Combinatorica}, 23(1):5--22, 2003.

\bibitem{ashikhmin2005bounds}
A.~E. Ashikhmin, G.~D. Cohen, M.~Krivelevich, and S.~N. Litsyn.
\newblock Bounds on distance distributions in codes of known size.
\newblock {\em IEEE transactions on information theory}, 51(1):250--258, 2005.

\bibitem{bollobas1986combinatorics}
B.~Bollob{\'a}s.
\newblock {\em Combinatorics: set systems, hypergraphs, families of vectors,
  and combinatorial probability}.
\newblock Cambridge University Press, 1986.

\bibitem{bonin2003introduction}
J.~Bonin.
\newblock An introduction to extremal matroid theory with an emphasis on the
  geometric perspective (course notes).
\newblock {\em Universitat Politecnica de Catalunya, Barcelona}, 2003.

\bibitem{bonisoli1984every}
A.~Bonisoli.
\newblock Every equidistant linear code is a sequence of dual hamming codes.
\newblock {\em Ars Combinatoria}, 18:181--186, 1984.

\bibitem{briggs2017inverting}
J.~Briggs and C.~Cox.
\newblock Inverting the {T}ur{\'a}n problem.
\newblock {\em arXiv preprint arXiv:1711.02082}, 2017.

\bibitem{cooper2016minors}
C.~Cooper, A.~Frieze, and W.~Pegden.
\newblock Minors of a random binary matroid.
\newblock {\em arXiv preprint arXiv:1612.02084}, 2016.

\bibitem{hamming1950error}
R.~W. Hamming.
\newblock Error detecting and error correcting codes.
\newblock {\em Bell Labs Technical Journal}, 29(2):147--160, 1950.

\bibitem{kramer2010most}
J.~B. Kramer.
\newblock On the most weight $ w $ vectors in a dimension $ k $ binary code.
\newblock {\em The Electronic Journal of Combinatorics}, 17(1):142, 2010.

\end{thebibliography}

\end{document}